\documentclass{birkjour}
 \newtheorem{thm}{Theorem}[section]
 \newtheorem{cor}[thm]{Corollary}
 \newtheorem{lem}[thm]{Lemma}

 \theoremstyle{definition}
 \newtheorem{defn}[thm]{Definition}
 \theoremstyle{remark}
 \newtheorem{rem}[thm]{Remark}
 \newtheorem*{ex}{Example}

 \usepackage{amsmath}
\usepackage{amsfonts}
\usepackage{amssymb}
 \usepackage[mathscr]{eucal}
 \usepackage[cp1251]{inputenc}
\usepackage[OT1,T1,TS1,T2A]{fontenc}
\usepackage[russian,english]{babel}
\numberwithin{equation}{section}
\newcommand{\ccomma}{\mathpunct{\raisebox{0.5ex}{,}}}

\begin{document}

\title[On the roots of a hyperbolic polynomial pencil]
 {On the roots of a hyperbolic polynomial pencil}

\author[Victor Katsnelson]{Victor Katsnelson}

\address{%
Department of Mathematics, Weizmann Institute, Rehovot, 7610001, Israel}

\email{victor.katsnelson@weizmann.ac.il; victorkatsnelson@gmail.com}




\date{January 1, 2016}
\dedicatory{}

\begin{abstract}{\ \ }
Let \(\nu_0(t),\nu_1(t),\,\ldots\,,\nu_n(t)\) be the roots of the
equation \(R(z)=t\), where \(R(z)\) be a rational function of the form
\[R(z)=z-\sum\limits_{k=1}^n\frac{\alpha_k}{z-\mu_k},\]
\(\mu_k\) are pairwise different real numbers, \(\alpha_k>0,\,1\leq{}k\leq{}n\).
Then for each real \(\xi\), the function
\(e^{\xi\nu_0(t)}+e^{\xi\nu_1(t)}+\,\cdots\,+e^{\xi\nu_n(t)}\) is exponentially convex
on the interval \(-\infty<t<\infty\).
\end{abstract}
\maketitle
\section{Roots of the equation \(\boldsymbol{R(z)=t}\) as functions of \(\boldsymbol{t}\).}
In the present paper we discuss questions related to properties of roots of the equation
\begin{equation}
\label{Eqt}
R(z)=t
\end{equation}
as functions of the parameter \(t\in\mathbb{C}\),
where \(R\) is a rational function of the form
\begin{equation}
\label{RaF}
R(z)=z-\sum\limits_{1\leq k \leq n}\frac{\alpha_{k}}{z-\mu_{k}}\ccomma
\end{equation}
\(\mu_k\) are pairwise different real numbers, \(\alpha_k>0\),
 \(1\leq k\leq n\).  We adhere to
 the enumeration agreement%
 \footnote{We assume that \(n\geq1\).}
 \begin{equation}
 \label{EnAg}
 \mu_1>\mu_2>\,\cdots\,>\mu_n.
 \end{equation}

The function \(R\) is representable in the form
\begin{equation}
 \label{RR}
 R(z)=\frac{P(z)}{Q(z)},
 \end{equation}
 where
 \begin{gather}
 \label{QDe}
 Q(z)=(z-\mu_1)\cdot(z-\mu_2)\cdot\,\cdots\,\cdot(z-\mu_n),\\
 \label{PDe}
 P(z)\stackrel{\textup{def}}{=}R(z)\cdot{}Q(z)
 \end{gather}
 are monic polynomials of degrees
 \begin{equation}
 \label{Deg}
 \textup{deg}\,P=n+1,\quad \textup{deg}\,Q=n.
 \end{equation}
 Since \(P(\mu_k)=-\alpha_k Q^{\prime}(\mu_k)\not=0\), the polynomials \(P\) and \(Q\) have no common roots. Thus the ratio in the right hand side of \eqref{RR} is irreducible.
 The equation \eqref{Eqt} is equivalent to the equation
 \begin{equation}
 \label{EqEq}
 P(z)-tQ(z)=0.
 \end{equation}
 Since the polynomial \(P(z)-tQ(z)\) is of degree \(n+1\), the latter equation has \(n+1\) roots for each \(t\in\mathbb{C}\).

 The function \(R\) possess the property
 \begin{equation}
 \label{NePr}
 \textup{Im}\,R(z)\big/\textup{Im}\,z>0 \ \  \textup{if}\ \textup{Im}\,z\not= 0.
 \end{equation}
  Therefore if \(\textup{Im}\,t>0\), all roots of the equation
 \eqref{Eqt}, which is equivalent to the equation \eqref{EqEq}, are located in the half-plane \(\textup{Im}\,z>0\). Some of these roots may be multiple.

 However if \(t\) is real, all roots of the equation \eqref{Eqt} are real and simple,
 i.e. of multiplicity one. Thus for real \(t\), the equation \eqref{Eqt} has \(n+1\)
 pairwise different real roots
 \(\nu_k(t)\): \(\nu_0(t)>\nu_1(t)>\,\cdots\,>\nu_{n-1}(t)>\nu_n(t)\).
 Moreover for each real \(t\), the poles \(\mu_k\) of the function \(R\) and the roots \(\nu_k(t)\) of the equation \eqref{Eqt} are interspersed:
 \begin{multline}
\label{InSp}
\nu_0(t)>\mu_1>\nu_1(t)>\mu_2>\nu_2(t)>\,\,\cdots\,\,>\nu_{n-1}(t)>\mu_n>\nu_{n}(t),\\
\ \forall\,t\in\mathbb{R}.
 \end{multline}
 In particular for \(t=0\), the roots \(\nu_k(0)=\lambda_k\) of the equation \eqref{Eqt} are the roots
 of the polynomial \(P\):
\begin{gather}
\label{PRo}
P(z)=(z-\lambda_0)\cdot(z-\lambda_1)\cdot\,\,\cdots\,\,\cdot(z-\lambda_n),\\
\lambda_0>\mu_1>\lambda_1>\mu_2>\lambda_2>\,\,\cdots\,\,>\lambda_{n-1}>\mu_n>\lambda_{n}.
\end{gather}
Since \(R^{\prime}(x)>0\) for \(x\in\mathbb{R},\,x\not=\mu_1,\,\ldots\,,\mu_n\),
each of the functions \(\nu_k(t),\,k=0,1,\,\ldots\,,n\), can be continued as
a single valued holomorphic function to some neighborhood of \(\mathbb{R}\).
However the functions \(\nu_k(t)\) can not be continued as single-valued
analytic functions to the whole complex \(t\)-plane.  According to \eqref{RR},
\begin{equation}
\label{DR}
R^{\prime}(z)=\frac{P^{\prime}(z)Q(z)-Q^{\prime}(z)P(z)}{Q^2(z)}.
\end{equation}
The polynomial \(P^{\prime}Q-Q^{\prime}P\) is of degree \(2n\)
and  is strictly positive on the real axis. Therefore this
polynomial has \(n\) roots \(\zeta_1,\,\ldots\,,\zeta_n\)
in the upper half-plane \(\textup{Im}(z)>0\) and \(n\) roots
\(\overline{\zeta_1},\,\ldots\,,\overline{\zeta_n}\)
in the lower half-plane \(\textup{Im}(z)<0\).
(Not all roots \(\zeta_1,\,\ldots\,,\zeta_n\) must be different.)
The points \(\zeta_1,\,\ldots\,,\zeta_n\) and \(\overline{\zeta_1},\,\ldots\,,\overline{\zeta_n}\) are the critical points of the function \(R\):
\(R^{\prime}(\zeta_k)=0,\,R^{\prime}(\overline{\zeta_k})=0,\ 1\leq k\leq n.\)
The critical values \(t_k=R(\zeta_k),\,\overline{t_k}=R(\overline{\zeta_k}),\ 1\leq k\leq n,\) of the function \(R\) are the ramification points of
the function \(\nu(t)\):
\begin{equation}
\label{RoF}
R(\nu(t))=t
\end{equation}
 (Even if the critical points \(\zeta^{\prime}\) and \(\zeta^{\prime\prime}\) of
 \(R\) are different, the critical values \(R(\zeta^{\prime})\) and
 \(R(\zeta^{\prime\prime})\) may coincide.)
 We denote the set of critical values of the function \(R\) by \(\mathcal{V}\) :
\begin{equation}
\label{CrV}
\mathcal{V}=\mathcal{V}^{+}\cup\mathcal{V}^{-},\quad \mathcal{V}^{+}=\{t_1,\,\ldots\,,t_n\},\ \mathcal{V}^{-}=\{\overline{t_1},\,\ldots\,,\overline{t_n}\}.
\end{equation}
Not all values \(t_1,\,\ldots\,,t_n\) must be different.
However \(\mathcal{V}\not=\emptyset\).
In view of \eqref{NePr}, \(\textup{Im}\,t_k>0,\,1\leq k\leq n\). So
\begin{equation}
\label{CrVa}
\mathcal{V}^{+}\subset\{t\in\mathbb{C}:\,\textup{Im}\,t>0\},\quad
\mathcal{V}^{-}\subset\{t\in\mathbb{C}:\,\textup{Im}\,t<0\}.
\end{equation}

Ler \(G\) be an arbitrary simply connected domain in the \(t\)-plane which does not
intersect the set \(\mathcal{V}\). Then the roots of equation \eqref{Eqt} are pairwise
different for each \(t\in{}G\). We can enumerate these roots, say \(\nu_0(t),\nu_1(t),\,\ldots\,\nu_n(t)\), such that all functions \(\nu_k(t)\)
are holomorphic in \(G\).

  The strip \(S_h\),
  \begin{equation}
  \label{Str}
  S_h=\{t\in\mathbb{C}:|\textup{Im}\,t<h|\},\ \ \textup{where} \ \ h=\min\limits_{1\leq k\leq_n}\!\textup{Im}\,t_k,
  \end{equation}
   does not intersect the set \(\mathcal{V}\). So  \(n+1\) single valued holomorphic branches of the function \(\nu(t)\), \eqref{RoF}, are defined in the strip \(S_h\). \emph{We choose such enumeration of
these branches  which agrees with the enumeration \eqref{InSp} on \(\mathbb{R}\).}
The set \(\{\}\).

Let \(L\) be a Jordan curve in \(\mathbb{C}\) which possess the properties:
\begin{enumerate}
\item \(L\subset\{t\in\mathbb{C}: \textup{Im}\,t>-h\}\) ;
\item The  set \(\mathcal{V}^{+}\) is contained in the interior of the curve \(L\).
\item \(L\cap\mathbb{R}\not=\emptyset\).
\end{enumerate}
Let us choose and fix a point \(t_0\in L\cap \mathbb{R}\). We consider the curve \(L\)
as a loop with base point \(t_0\) oriented counterclockwise.
Each branch \(\nu_k\) of the function \(\nu(t)\), \eqref{RoF}, can be continued analytically along \(L\) from a small neighborhood of the point \(t_0\) considered as an initial point of the loop \(L\) to the same neighborhood of the point \(t_0\) but considered as a final point of this loop. Continuing analytically the branch indexed as \(\nu_k\), we come to the branch indexed
as \(\nu_{k-1}\), \(k=0,1,\,\ldots\,,n\). (We put \(\nu_{-1}\stackrel{\textup{\tiny def}}{=}\nu_n\).)

From \eqref{PDe} and \eqref{RaF} it follows that the polynomial \(P\) is
representable in the form
\begin{subequations}
\label{PQa}
\begin{gather}
\label{Pa}
P(z)=z\,Q(z)-\sum\limits_{k=1}^{n}\alpha_kQ_k(z),\\
\intertext{where}
\label{Qa}
Q_k(z)=Q(z)/(z-\mu_k),\quad k=1,2,\,\ldots\,,n.
\end{gather}
\end{subequations}
\section{Determinant representation of the polynomial pencil \(\boldsymbol{P(z)-tQ(z)}\).}

The polynomial pencil \(P\) is \emph{hyperbolic}: for each real \(t\), all
roots of the equation \eqref{EqEq} are real.

Using \eqref{PQa}, we represent the polynomial \(P(z)-tQ(z)\) as the characteristic polynomial \(\det(zI-(A+tB))\) of some matrix pencil,
where \(A\) and \(B\) are self-adjoint \((n+1)\times(n+1)\)
matrices, \(\textup{rank}\,B=1\). We present these matrices explicitly.
\begin{lem}
\label{LDeRe}
Let  \(A=\|a_{p,q}\|\) and \(B=\|b_{p,q}\|\), \(0\leq{}p,q\leq{}n\), be \((n+1)\times(n+1)\) matrices with the entries
\begin{gather}
a_{0,0}=0, \ a_{p,p}=\mu_p \ \textup{for} \ p=1,2,\,\ldots\,,n, \ \
 \notag\\
 a_{p,q}=0 \ \textup{for} \ p=1,2,\,\ldots\,,n,\ q=1,2,\,\ldots\,,n, \
 p\not=q, \label{MatA} \\
 a_{0,p}=\overline{a_{p,0}} \ \textup{for} \
 p=1,2,\,\ldots\,,n, \notag
\end{gather}
and
\begin{gather}
\label{MatB}
b_{0,0}=1,\  \textup{all other} \ b_{p,q} \ \textup{vanish.}
\end{gather}
Then the equality
\begin{equation}
\label{DeRe}
\det(zI-A-tB)=(z-t)\cdot{}Q(z)-\sum\limits_{k=1}^n|a_{k,n+1}|^2Q_k(z).
\end{equation}
holds.
\end{lem}
\begin{proof} The matrix \(zI-(A+tB)\) is of the form
\begin{equation*}
zI-(A+tB)=
\begin{bmatrix}
z-t&-a_{0,1}&-a_{0,2}&\cdots&-a_{0,n-1}&-a_{0,n}\\[1.2ex]
-\overline{a_{0,1}}&z-\mu_1&0&\cdots&0&0\\[1.2ex]
-\overline{a_{0,2}}&0&z-\mu_2&\cdots&0&0\\[1.2ex]
\hdotsfor[1.8]{6}\\[1.2ex]
-\overline{a_{0,n-1}}&0&0&\cdots&z-\mu_{n-1}&0\\[1.2ex]
-\overline{a_{0,n}}&0&0&\cdots&0&z-\mu_n
\end{bmatrix}
\end{equation*}
We compute the determinant of this matrix using the cofactor formula.
\end{proof}
Comparing \eqref{PQa} and \eqref{DeRe}, we see that if the conditions
\begin{equation}
\label{CruCo}
|a_{0,p}|^2=\alpha_p,\quad p=1,2,\,\ldots\,,n
\end{equation}
are satisfied, then the equality
\begin{equation}
\label{CruEq}
P(z)-tQ(z)=\det(zI-A-tB)
\end{equation}
holds for every \(z\in\mathbb{C}, t\in\mathbb{C}\).

The following result is an immediate consequence of Lemma \ref{LDeRe}.
\begin{thm}
\label{TDeRe}
Let \(R\) be a function of the form \eqref{RaF}, where \(\mu_1,\mu_2,\,\ldots\,,\mu_n\) are pairwise different real numbers and
\(\alpha_1,\alpha_2,\,\ldots\,,\alpha_n\) are positive numbers.
Let \(Q\) and \(P\) be the polynomials related to the the function \(R\) by
the equalities \eqref{QDe} and \eqref{PQa}.

Then the pencil of polynomials \(P(z)-tQ(z)\) is representable as the characteristic polynomial of the matrix pencil \(A+tB\), i.e. the equality \eqref{CruEq} holds for every \(z\in\mathbb{C}, t\in\mathbb{C}\), where \(B\) is the matrix with the entries \eqref{MatB}, and the entries of the matrix \(A\) are defined by
by \eqref{MatA} with
\begin{equation}
\label{UpRo}
a_{0,p}=\sqrt{\alpha_p}\,\omega_p, \quad p=1,2,\,\ldots\,,n,
\end{equation}
\(\omega_p\) are arbitrary%
\footnote{We will use the freedom in choosing \(\omega_p\) to prescribe signs \(\pm\) to the entries \(a_{0,p}\).
}
 complex numbers which absolute value equals one:
\begin{equation}
\label{Uni}
|\omega_p|=1,\quad p=1,2,\,\ldots\,,n.
\end{equation}
\end{thm}
\begin{cor}
\label{cor}
Let \(R, A, B\) be the same that in Theorem \ref{TDeRe}. For each \(t\in\mathbb{C}\), the roots \(\nu_0(t),\nu_0(t),\,\ldots\,,\nu_n(t)\) of the equation \eqref{RaF} are the eigenvalues
of the matrix \(A+tB\).
\end{cor}
\begin{lem}
\label{Trf}
Let \(R, A, B\) be the same that in Theorem \ref{TDeRe}, \(\nu_0(t),\nu_0(t),\,\ldots\,,\) \(\nu_n(t)\) be the roots of the equation \eqref{RaF} and \(h(z)\) be an entire function. Then the equality
\begin{equation}
\label{TrEqu}
\sum\limits_{k=0}^nh(\nu_k(t))=\textup{trace}\,\big\{h(A+tB)\big\}
\end{equation}
holds for every \(t\in\mathbb{C}\).
\end{lem}
\begin{proof} We refer to Corollary \ref{cor}.
If \(\nu\) is an eigenvalue of some square  matrix \(M\), then \(h(\nu)\)
is an eigenvalue of the matrix \(h(M)\). In \eqref{TrEqu}, we interpret the trace
of the matrix \(h(A+tB)\) as its \emph{spectral trace}, that is as the sum
of all its eigenvalues.
\end{proof}

\section{Exponentially convex functions.}
 \begin{defn}
 \label{decf}
A function \(f(t)\) on the interval \(a<t<b\) is said to be belongs to the class
 \(W_{a,b}\) if \(f\) is continuous on \((a,b)\)
and if  all forms
\begin{equation}
\label{pqf}
\sum\limits_{r,s=1}^{N}f(t_r+t_s)\zeta_r\overline{\zeta_s}\quad (N=1,2,3,\ldots\,)
\end{equation}
are non-negative for every choice of complex numbers \(\zeta_1,\zeta_2,\,\ldots\,,\zeta_N\) and for every choice of real numbers \(t_1,t_2,\,\ldots\,,t_N\) assuming that all sums \(t_r+t_s\) are within the interval \((a,b)\).
\end{defn}
The class \(W_{a,b}\) was introduced by S.N.Bernstein, \cite{Be}, see \S 15 there.
Somewhat later, D.V.Widder also introduced the class \(W_{a,b}\) and studied it. S.N.Bernstein call functions \(f(x)\in{}W_{a,b}\) exponentially convex.

\noindent
\begin{center}
\textbf{Properties of the class of exponentially convex functions.}
\end{center}
\begin{enumerate}
\item[\textup{P\,1}.] If \(f(t)\in{}W_{a,b}\) and \(c\geq0\) is a nonnegative constant, then \(cf(t)\in{}W_{a,b}\).
\item[\textup{P\,2}.] If \(f_1(t)\in{}W_{a,b}\) and \(f_2(t)\in{}W_{a,b}\), then
\(f_1(t)+f_2(t)\in{}W_{a,b}\).
\item[\textup{P\,3}.] If \(f_1(t)\in{}W_{a,b}\) and \(f_2(t)\in{}W_{a,b}\), then
\(f_1(t)\cdot f_2(t)\in{}W_{a,b}\).
\item[\textup{P\,4}.] Let \(\lbrace f_{n}(t)\rbrace_{1\leq n<\infty}\) be a sequence of functions from the class \(W_{a,b}\). We assume that for each \(t\in(a,b)\) there exists the limit
\(f(t)=\lim_{n\to\infty}f_{n}(t)\), and that \(f(t)<\infty\ \forall t\in(a,b)\).
Then \(f(t)\in{}W_{a,b}\).
\end{enumerate}

From the functional equation for the exponential function
it follows that for each real number \(u\), for every choice of real numbers \(t_1,t_2,\,\ldots\,\), \(t_{N}\) and complex numbers
\(\zeta_1\), \(\zeta_2, \,\ldots\,, \zeta_{N}\), the equality holds
\begin{equation}
\label{ece}
\sum\limits_{r,s=1}^{N}e^{(t_r+t_s)\xi}\zeta_r\overline{\zeta_s}=
\bigg|\sum\limits_{p=1}^{N}e^{t_p{\xi}}\zeta_p\,\bigg|^{\,2}\geq 0.
\end{equation}
The relation \eqref{ece} can be formulated as
\begin{lem}
\label{ECE}
For each real number \(\xi\), the function \(e^{t\xi}\) of the variable \(t\) belongs to the class \(W_{-\infty,\infty}\).
\end{lem}
The term \emph{exponentially convex function} is justified by an integral
representation for any function \(f(t)\in{}W_{a,b}\).
\begin{thm}[The representation theorem]
\label{RepTe}
In order the representation
\begin{equation}
\label{IRep}
f(x)=\int\limits_{\xi\in(-\infty,\infty)}e^{\xi{}x}\sigma(d\xi)\quad(a<x<b)
\end{equation}
  be valid, where \(\sigma(d\xi)\) is a non-negative measure, it is necessary and
  sufficient that \(f(x)\in{}W_{a,b}\).
\end{thm}
The proof of the Representation Theorem can be found in \cite{A},\,Theorem 5.5.4, and in \cite{W},\,Theorem 21.
\begin{cor}
\label{Cor}
The representation \eqref{IRep} shows that \(f(x)\) is the value of a function \(f(z)\)
holomorphic in the strip \(a<\textup{Re}\,z<b\).
\end{cor}

\section{Herbert Stahl's Theorem.}
In the paper \cite{BMV} a conjecture was
formulated which now is commonly known as the BMV conjecture:\\[1.0ex]
\textbf{The BMV Conjecture.} \textit{Let \(U\) and \(V\) be Hermitian matrices of size
\(l\times{}l\).
Then the function
\begin{equation}
\label{TrF}
\varphi(t)=
\textup{trace}\,\big\{e^{U+tV}\big\}
\end{equation}
of the variable \(t\) belongs to the class \(W_{-\infty,\infty}\).}

If the matrices \(U\) and \(V\) commute, the exponential convexity of the function
\(\varphi(t)\), \eqref{TrF}, is evident. In this case, the sum
\[\sum\limits_{r,s=1}^{N}\varphi(t_r+t_s)\xi_r\overline{\xi_s}=
\textup{trace}\,\bigg\{e^{U/2}\Big(\sum\limits_{r=1}^{N}e^{t_rV}\xi_r\Big)
\Big(\sum\limits_{s=1}^{N}e^{t_sV}\xi_s\Big)^{\ast}\big(e^{U/2}\big)^{\ast}\bigg\}\]
is non-negative because this sum is the trace of a non-negative matrix. The measure
\(\sigma\) in the integral representation \eqref{IRep} of the function \(\varphi(t)\),
\eqref{TrF}, is an atomic measure supported on the spectrum of the matrix \(V\).

In general case, if the matrices \(U\) and \(V\) do not commute, the BMV conjecture remained an
open question for longer than 40 years. In 2011, Herbert
Stahl gave an affirmative answer to the BMV conjecture.
\begin{thm}
\label{HStT}
\textup{(H.Stahl)} Let \(U\) and \(V\) be  Hermitian matrices of size \(l\times{}l\).

Then the function \(\varphi(t)\) defined by \eqref{TrF} belongs to the class \(W_{-\infty,\infty}\) of functions exponentially convex on \(-\infty,\infty\).
\end{thm}
The first arXiv version of H.Stahl's Theorem appeared in \cite{S1}, the latest arXiv version -
in \cite{S2}, the journal publication - in \cite{S3}.

The proof of Herbert Stahl is based on ingenious considerations related to
Riemann surfaces of algebraic functions. In \cite{E}, a simplified version of the Herbert Stahl proof is presented.

We present a toy version of Theorem \ref{HStT} which is enough for our goal.
\begin{thm}
\label{toy}
Let \(U\) and \(V\) be  Hermitian matrices of size \(l\times{}l\). We assume moreover
that
\begin{enumerate}
\item
All off-diagonal entries of the matrix \(U\) are non-negative.
\item
The matrix \(V\) is diagonal.
\end{enumerate}
 Then the function \(\varphi(t)\) defined by \eqref{TrF} belongs to the class \(W_{-\infty,\infty}\).
\end{thm}
\begin{proof}
For \(\rho\geq 0\), let \(U_{\rho}=U+\rho{}I\), where \(I\) is the identity matrix.
If \(\rho\) is large enough, then all entries of the matrix \(U_{\rho}\) are non-negative.
Let us choose and fix such \(\rho\).
It is clear that
\begin{equation}
\label{PrF}
e^{U+tV}=e^{-t\rho}\,e^{U_{\rho}+tV}.
\end{equation}
We use the Lie product formula
\begin{equation}
\label{LPF}
e^{U_{\rho}+tV}=\lim_{m\to\infty}\Big(e^{U_{\rho}/m}\,e^{tV/m}\Big)^m.
\end{equation}
All entries of the matrix \(e^{U_{\rho}/m}\) are non-negative numbers.
Since matrix \(V\) is Hermitian, its diagonal entries are real numbers. Thus \[e^{tV/m}=\textup{diag\,(}e^{tv_1/m},e^{tv_2/m},\,\ldots\,e^{tv_m/m}\textup{)},\]
where \(v_1,v_2,\,\ldots\,,v_m\) are real numbers. The exponentials \(e^{tv_j/m}\)
are functions of \(t\) from the class \(W_{-\infty,\infty}\).
Each entry of the matrix \(e^{U_{\rho}/m}\,e^{tV/m}\) is a linear combination of these exponentials with non-negative coefficients. According to the properties P1 and P2 of the class \(W_{-\infty,\infty}\), the entries of the matrix
\(e^{U_{\rho}/m}\,e^{tV/m}\) are functions of the class \(W_{-\infty,\infty}\).
Each entry of the matrix \(\Big(e^{U_{\rho}/m}\,e^{tV/m}\Big)^m\) is a sum
of products of some entries of the matrix \(e^{U_{\rho}/m}\,e^{tV/m}\).
According to the properties P2 and P3 of the class \(W_{-\infty,\infty}\),
the entries of the matrix \(\Big(e^{U_{\rho}/m}\,e^{tV/m}\Big)^m\)
are functions of \(t\) belonging to the class \(W_{-\infty,\infty}\). From the limiting relation
\eqref{LPF} and from the property P4 of the class \(W_{-\infty,\infty}\)
it follows that all entries of the matrix \(e^{U_{\rho}+tV}\) are function
of \(t\) belonging to the class \(W_{-\infty,\infty}\). From \eqref{PrF}  it follows that all entries of the matrix \(e^{U+tV}\) belong
to the class \(W_{-\infty,\infty}\). All the more, the function
\(\varphi(t)=\textup{trace}\,\big\{e^{U+tV}\big\}\),
which is the sum of diagonal entries of the matrix \(e^{U+tV}\), belongs
to the class \(W_{-\infty,\infty}\).
\end{proof}

\section{Exponential convexity of the sum
\(\boldsymbol{e^{\xi\nu_{0}(t)}\,+\ldots+\,e^{\xi\nu_{n}(t)}.}\)}
Let \(\xi\) be a real number. Taking \(h(z)=e^{\xi{}z}\) in Lemma \ref{Trf},
we obtain
\begin{lem}
\label{AtE}
Let \(R\) be the rational function of the form \eqref{RaF},
\(\nu_{0}(t),\nu_1(t),\,\ldots\,,\) \(\nu_n(t)\) be the roots of the equation
\eqref{Eqt}. Let \(A\) and \(B\) be the matrices \eqref{MatA}, \eqref{UpRo},
\eqref{MatB} which appear in the determinant representation \eqref{CruEq}
of the matrix pencil \(P(z)-tQ(z)\).

Then the equality
\begin{equation}
\label{trEqu}
\sum\limits_{k=0}^{n}e^{\xi\,\nu_k(t)}=\textup{trace}
\big\{e^{\xi{}A+t(\xi{}B)}\big\}
\end{equation}
holds.
\end{lem}
Now we choose \(\omega_p\) in \eqref{UpRo} so that all off-diagonal entries of
the matrix \(U=\xi{}A\) are non-negative: if \(\xi>0\), then \(\omega_p=+1\), if
\(\xi<0\), then \(\omega_p=-1\), \(1\leq p\leq n\).

Applying Theorem \ref{toy} to the matrices \(U=\xi{}A, V=\xi{}B\), we obtain
the following result
\begin{thm}
\label{Eeco}
Let \(R\) be the rational function of the form \eqref{RaF},
\(\nu_{0}(t),\nu_1(t),\) \(\,\ldots\,, \nu_n(t)\) be the roots of the equation
\eqref{Eqt}.  Then for each \(\xi\in\mathbb{R}\), the function
\begin{equation}
g(t,\xi) \stackrel{\textup{\tiny def}}{=}\sum\limits_{k=0}^{n}e^{\xi\,\nu_k(t)}
\end{equation}
of the variable \(t\) belongs to the class \(W_{-\infty,\infty}\).
\end{thm}
\begin{thm}
\label{Su}
Let
\begin{math}
f\in{}W_{u,v},\  \text{where} \ -\infty\leq{}u<v\leq+\infty.
 \end{math}
 Let \(R\) be the rational function of the form \eqref{RaF},
\(\nu_{0}(t),\nu_1(t),\) \(\,\ldots\,, \nu_n(t)\) be the roots of the equation
\eqref{Eqt}. Assume that for some \(a,b\), \(-\infty\leq{}a<b\leq+\infty\), the inequalities
\begin{equation}
\label{Ine}
u<\nu_k(t)<v,\quad a<t<b,\ \ k=0,1,\,\ldots\,,n
\end{equation}
hold.

Then the function
\begin{equation}
\label{SuSu}
F(t)\stackrel{\textup{\tiny def}}{=}\sum\limits_{k=0}^{n}f(\nu_k(t))
\end{equation}
belongs to the class \(W_{a,b}\).
\end{thm}
\begin{proof} According to Theorem \ref{RepTe}, the representation
 \begin{equation*}
 f(x)=\int\limits_{\xi\in(-\infty,\infty)}e^{\xi{}x}\sigma(d\xi), \ \ \forall\,x\in(u,v)
 \end{equation*}
  holds, where \(\sigma\) is a non-negative measure.
 Substituting \(x=\nu_k(t)\) to the above formula, we obtain the equality
 \begin{equation*}
 f(\nu_k(t))=\int\limits_{\xi\in(-\infty,\infty)}e^{\xi\nu_k(t)}\sigma(d\xi), \ \
 \forall\,t\in(a,b),\ \ k=0,1,\,\ldots\,,n.
 \end{equation*}
 Hence
 \begin{equation}
 F(t)=\int\limits_{\xi\in(-\infty,\infty)}g(t,\xi)\,\sigma(d\xi), \ \ \forall\,t\in(a,b).
 \end{equation}
 Theorem \ref{SuSu} is a consequence of Theorem \ref{Eeco} and of the properties
 P1,P2,P4 of the class of exponentially convex functions.
\end{proof}
\noindent
\textbf{Example} For \(\gamma>0\), the function \(f(x)=e^{\gamma{}x^2}\) is exponentially convex on \((-\infty,\infty)\):
\(e^{\gamma{}x^2}=\int\limits_{\xi\in(-\infty,\infty)}e^{\xi{}x}\sigma(d\xi), \ \
\textup{where}
\ \ \sigma(d\xi)=\frac{1}{2\sqrt{\pi\gamma}}e^{-\xi^2/4\gamma}d\xi.\)\\
Thus the function \(F(t)=\sum\limits_{k=0}^ne^{\gamma(\nu_k(t))^2}\)
is exponentially convex on \((-\infty,\infty)\).
\begin{rem}
Familiarizing himself with our proof of Theorem \ref{Eeco}, Alexey Kuznetsov
(\,\,www.math.yorku.ca/\raisebox{-0.8ex}{\~}{}akuznets/\,\,)
 gave a new proof of a somewhat weakened version of this theorem. His proof is based on the theory of stochastic processes L\'{e}vy.
\end{rem}

\end{document}